\documentclass[12pt,reqno,a4paper]{amsart}
\usepackage{graphicx}
\usepackage{amssymb,amscd,textcomp}

\usepackage{hyperref}

\DeclareFontFamily{U}{mathb}{\hyphenchar\font45}
\DeclareFontShape{U}{mathb}{m}{n}{
      <5> <6> <7> <8> <9> <10> gen * mathb
      <10.95> mathb10 <12> <14.4> <17.28> <20.74> <24.88> mathb12
}{}
\DeclareSymbolFont{mathb}{U}{mathb}{m}{n}
\DeclareMathSymbol{\llcurly}{3}{mathb}{"CE}

\begin{document}

\let\kappa=\varkappa
\let\eps=\varepsilon
\let\phi=\varphi
\let\p\partial
\let\lle=\preccurlyeq
\let\ulle=\curlyeqprec

\def\Z{\mathbb Z}
\def\R{\mathbb R}
\def\C{\mathbb C}
\def\Q{\mathbb Q}
\def\P{\mathbb P}
\def\HH{\mathsf{H}}
\def\XX{\mathcal{X}}

\def\conj{\overline}
\def\Beta{\mathrm{B}}
\def\const{\mathrm{const}}
\def\ov{\overline}
\def\wt{\widetilde}
\def\wh{\widehat}

\renewcommand{\Im}{\mathop{\mathrm{Im}}\nolimits}
\renewcommand{\Re}{\mathop{\mathrm{Re}}\nolimits}
\newcommand{\codim}{\mathop{\mathrm{codim}}\nolimits}
\newcommand{\Aut}{\mathop{\mathrm{Aut}}\nolimits}
\newcommand{\lk}{\mathop{\mathrm{lk}}\nolimits}
\newcommand{\sign}{\mathop{\mathrm{sign}}\nolimits}
\newcommand{\rk}{\mathop{\mathrm{rk}}\nolimits}

\def\id{\mathrm{id}}
\def\Leg{\mathrm{Leg}}
\def\Jet{{\mathcal J}}
\def\sS{{\mathcal S}}
\def\lcan{\lambda_{\mathrm{can}}}
\def\ocan{\omega_{\mathrm{can}}}
\def\bgamma{\boldsymbol{\gamma}}

\renewcommand{\mod}{\mathrel{\mathrm{mod}}}

\newtheorem{mainthm}{Theorem}
\renewcommand{\themainthm}{{\Alph{mainthm}}}
\newtheorem{thm}{Theorem}[section]
\newtheorem{lem}[thm]{Lemma}
\newtheorem{prop}[thm]{Proposition}
\newtheorem{cor}[thm]{Corollary}

\theoremstyle{definition}
\newtheorem{exm}[thm]{Example}
\newtheorem{rem}[thm]{Remark}
\newtheorem{df}[thm]{Definition}
\newtheorem*{que}{Question}

\numberwithin{equation}{section}

\title{Interval topology in contact geometry}
\author[Chernov \& Nemirovski]{Vladimir Chernov and Stefan Nemirovski}
\thanks{This work was partially supported by a grant from the Simons Foundation (\#\,513272 to Vladimir Chernov).
The second author was partially supported by CRC/TRR~191 of the DFG and by RFBR grant \textnumero 17-01-00592-a}
\address{Department of Mathematics, 6188 Kemeny Hall,
Dartmouth College, Hanover, NH 03755-3551, USA}
\email{Vladimir.Chernov@dartmouth.edu}
\address{%
Steklov Mathematical Institute, Gubkina 8, 119991 Moscow, Russia;\hfill\break
\strut\hspace{8 true pt} Mathematisches Institut, Ruhr-Universit\"at Bochum, 44780 Bochum, Germany}
\email{stefan@mi-ras.ru}

\begin{abstract}
A topology is introduced on spaces of Legendrian submanifolds and groups of contactomorphisms. 
The definition is motivated by the Alexandrov topology in Lorentz geometry.
\end{abstract}

\maketitle

\section{Introduction}
Let $(Y,\xi)$ be a contact manifold with a co-oriented 
contact structure. 
On every connected component $\mathcal{C}$ of the group 
of contactomorphisms of $Y$ or of the space of Legendrian 
submanifolds in $Y$,
there is a binary partial relation defined by setting
$$
a\llcurly b
$$
if there is a {\it positive\/} isotopy from $a$ to $b$. 
The family of intervals
$$
(a,b):=\{z\in\mathcal{C} \mid a\llcurly z\llcurly b\}
$$
with respect to this relation generates the 
{\it interval topology\/} on~$\mathcal{C}$. 

In Lorentz geometry, the topology generated by the family of intervals 
with respect to the chronology relation~$\ll$ on a spacetime $X$ was
considered in~\cite{Pi, KrPe} and became known
as the {\it Alexandrov topology}. 
For a sufficiently nice spacetime, the Alexandrov topology
is the pull-back of the interval topology 
on Legendrian spheres in the contact manifold of null geodesics~$\mathfrak{N}_X$
by the Penrose--Low twistor map~\cite{Lo1, Lo2} sending 
a point $x\in X$ to its celestial sphere~$\mathfrak{S}_x\subset\mathfrak{N}_X$,
see~\S\ref{Skies}. 

A basic question about the interval topology on $\mathcal{C}$
is whether or not it is Hausdorff. It is similar (to an extent)
to the non-degeneracy question for the Hofer distance~\cite{Ho} and
its descendants~\cite{Che2, Sh, RoZh}. 

An immediate observation is that the Hausdorff axiom is not satisfied if there exists a  
positive loop in~$\mathcal{C}$. If that loop is contractible,
the issue persists even after passing to the universal cover of $\mathcal{C}$. 
For instance, it follows from~\cite{Liu} that the interval topology 
can never be Hausdorff on the universal cover of the Legendrian 
isotopy class of a {\it loose\/} Legendrian submanifold.
Positive loops are known to be the only obstruction to
{\it orderability\/}~\cite{ElPo, ChNe3}, which suggests 
the following problem:

\begin{que}
Suppose that $\mathcal{C}$ is (universally) orderable (see \S\ref{Preoders}).
Is the interval topology Hausdorff on (the universal cover of)~$\mathcal{C}$?
\end{que}

The answer remains unknown in general. Using generating functions
methods~\cite{Vi, ChNe1, CFP}, we check that the 
interval topology is Hausdorff on the class 
of the zero section of the $1$-jet bundle of a closed manifold
and on the class of the fibre of the spherical cotangent bundle of a manifold
covered by an open subset of~$\R^n$. The latter case is relevant
for Lorentz geometry as it leads to a new causal completion 
for some globally hyperbolic spacetimes.

\subsection*{Organisation of the paper}
Section~\ref{DefsProps} introduces the interval topology 
and discusses its general properties. The case of $1$-jet
bundles and (certain) spherical cotangent bundles is dealt with
in Section~\ref{1jetsCoSph}. The last section explores
the relation to the Alexandrov topology on spacetimes.

\subsection*{Conventions}
All manifolds and maps are taken to be $C^\infty$-smooth.
Contactomorphisms of co-oriented contact structures 
are assumed to be co-orientation preserving. 

\section{Interval topology and orderability}
\label{DefsProps}

\subsection{Positive and non-negative isotopies in contact geometry}
\label{PosNneg}
Let $(Y,\ker\alpha)$ be a contact manifold with a co-oriented contact structure.
A Legendrian isotopy $\{L_t\}_{t\in [0,1]}$ in $(Y,\ker\alpha)$
is called {\it non-negative\/} if it has a parametrisation $\ell_t:L_0\to L_t$
such that
\begin{equation}
\label{DefLegNonneg}
\alpha\left(\tfrac{d}{dt}\ell_t(x)\right)\ge 0
\end{equation}
for all $t\in [0,1]$ and $x\in L_0$. If the inequality in~\eqref{DefLegNonneg}
is strict, the isotopy is said to be {\it positive}.

It is clear from the definition that the property to be positive
or non-negative does not depend on the choice of a parametrisation 
of the isotopy and of a contact form defining the (co-oriented) 
contact structure. This property is also obviously preserved by 
(co-orientation preserving) contactomorphisms of $(Y,\ker\alpha)$.

\begin{exm}
\label{PosInJets}
A neighbourhood of a Legendrian submanifold $L\subset Y$ is contactomorphic
to a neighbourhood of the zero section of the 1-jet bundle $\Jet^1(L)$
with its canonical contact form, see Subsec.~\ref{minimax}.
Every Legendrian that is sufficiently $C^1$-close to $L$
corresponds to the graph of the 1-jet of a smooth function on~$L$.
A small Legendrian isotopy of $L$ is non-negative (or positive)
if and only if the corresponding family of functions on $L$ is pointwise
non-decreasing (or increasing).
\end{exm}

Similarly, an isotopy of contactomorphisms
$\{\phi_t\}_{t\in [0,1]}$ is called non-negative if its contact Hamiltonian
\begin{equation}
\label{DefContNonneg}
H(\phi_t(x),t):=\alpha\left(\tfrac{d}{dt}\phi_t(x)\right)\ge 0
\end{equation}
for all $t\in [0,1]$ and~$x\in Y$. If the inequality is strict, the 
isotopy is called positive. This property is invariant with respect
to the left and right actions of the contactomorphism group on itself.
Note also that if $L\subset Y$ is a Legendrian submanifold and
$\{\phi_t\}$ is a non-negative (or positive) contact isotopy,
then $\{\phi_t(L)\}$ is a non-negative (or positive) Legendrian isotopy.

\begin{exm}
\label{Reeb}
The Reeb flow of any contact form is a positive contact isotopy.
(Its contact Hamiltonian with respect to that contact form is
identically equal to one.)
\end{exm}

\begin{rem}
\label{LegGraphs}
For a contactomorphism $\phi$ with $\phi^*\alpha=e^f\alpha$, let
$$
\Gamma_\phi = \{(x,\phi(x),-f(x))\in Y\times Y\times \R\mid x\in Y\}
$$
be its {\it Legendrian graph\/}. It is a Legendrian submanifold
in the contact manifold $\bigl(Y\times Y\times \R, \ker (e^u\pi_2^*\alpha - \pi_1^*\alpha)\bigr)$,
where $\pi_1$ and $\pi_2$ are the projections to the first and second factors
and $u$ is the coordinate on~$\R$.
A contact isotopy $\{\phi_t\}$ is non-negative or positive
if and only if the Legendrian isotopy $\{\Gamma_{\phi_t}\}$
is non-negative or positive in $Y\times Y\times\R$.
\end{rem}

\subsection{Partial (pre-)orders}
\label{Preoders}
Given two Legendrian submanifolds $L_1, L_2$ 
in a co-oriented contact manifold~$(Y,\ker\alpha)$,
we write
$$
L_1\lle L_2
$$
and
$$
L_1\llcurly L_2
$$
if there is a non-negative and, respectively, positive Legendrian isotopy 
from $L_1$ to $L_2$. 

For a pair of contactomorphisms 
$\phi_1,\phi_2\in\mathrm{Cont}(Y,\ker\alpha)$, we write
$$
\phi_1\lle \phi_2
$$
and
$$
\phi_1\llcurly \phi_2
$$
if there is a non-negative and, respectively, positive contact isotopy 
from $\phi_1$ to $\phi_2$.

If $\mathcal{C}$ is a connected component of the space of Legendrians
or of the contactomorphism group equipped with the usual $C^\infty$-topology,
then the relations $\lle$ and $\llcurly$ admit obvious lifts to the 
universal cover $\Pi:\widetilde{\mathcal{C}}\to \mathcal{C}$. 
Namely, $a\lle b$ and, respectively, $a\llcurly b$ for two elements 
$a,b\in \widetilde{\mathcal{C}}$ if there is a path connecting
$a$ to $b$ such that its projection to $\mathcal{C}$ is a
non-negative and, respectively, positive isotopy. (Note that
$\ulle$ was used to denote the lift of $\lle$ in~\cite{ChNe3}.)

The relations $\lle$ and $\llcurly$ are transitive. (This is 
obvious for $\lle$ and requires an easy interpolation argument 
for $\llcurly$, cf.\ e.g.~\cite[Lemma 2.2]{ChNe2} or ~\cite[Proof of Lemma 4.4]{ChNe3}.) 
It is also clear that $\lle$ is reflexive because a constant
isotopy is non-negative. 

\begin{df}
$\mathcal{C}$ is said to be {\it orderable\/} if $\lle$
is a partial order (i.e.\ if $\lle$ is antisymmetric)
and {\it universally orderable\/} if the lift of $\lle$
is a partial order on its universal cover $\widetilde{\mathcal{C}}$.
\end{df}

\begin{rem}
The notion of orderability in contact geometry was introduced 
by Eliashberg--Polterovich~\cite{ElPo} and Bhupal~\cite{Bh}. 
Varied terminology has been in use since then. For instance,
a closed contact manifold is orderable in the sense of~\cite{ElPo}
if the identity component of its contactomorphism group 
is universally orderable.
\end{rem}

\subsection{Interval topology on Legendrians}
Let $\mathcal{L}$ be an isotopy class of closed Legendrian submanifolds 
in a contact manifold $(Y,\ker\alpha)$.
The {\it interval topology\/} on $\mathcal{L}$ is defined by the family 
of {\it intervals\/} 
$$
I_{a,b} := (a,b):= \{z\in\mathcal{L} \mid a\llcurly z\llcurly b\},\quad a,b\in\mathcal{L}.
$$
The interval topology on the universal cover of $\mathcal{L}$ is defined 
in the same way using the lift of~$\llcurly$.

Note that intervals form a base for a topology. Indeed, 
every point $L\in\mathcal{L}$ lies in an interval between 
two $C^\infty$-close Legendrians. (For instance, one can take shifts 
of $L$ by the Reeb flow of a contact form.)
Furthermore, it follows from Example~\ref{PosInJets}
that if $L\in I_1\cap I_2$, then there is an interval 
of this type contained in $I_1\cap I_2$. 

The interval topology is obviously invariant with respect to the 
action of all contactomorphisms (i.e.\ not necessarily co-orientation 
preserving, as reversing $\llcurly$ does not affect intervals). 

It is easy to see from Example~\ref{PosInJets} that the interval topology 
is rougher than the $C^k$-topology for every $k\ge 1$ in the sense that 
its open sets are open in any smooth topology. 
Proposition~\ref{IntC0} implies that restricting the definition 
of the interval topology to a $C^0$-neighbourhood of a Legendrian 
submanifold equips its $C^1$-neighbourhood with a topology that is strictly
rougher than the $C^0$-topology on Legendrians.

\subsection{Hausdorff-ness and orderability} 
The property of the interval topology to be Hausdorff appears to be 
rather similar to orderability. There are two subtle points, however.
First, the interval topology is defined in terms of~$\llcurly$ whereas 
orderability is a property of~$\lle$.  This difficulty has been already
addressed in~\cite{ElPo} and~\cite{ChNe3}, which we are 
going to use now and in~\S\ref{ITopCont}. Secondly, the failure of the 
Hausdorff axiom does not formally imply the existence of non-negative
{\it loops}. We could only find a partial solution to this problem
in Proposition~\ref{disjoint}.

\begin{prop}
\label{Haus2order}
Let $\mathcal{L}$ be an isotopy class of closed Legendrian submanifolds. 
Suppose that the interval topology on the 
{\rm (}universal cover\/{\rm )} of\/ $\mathcal{L}$ is Hausdorff. 
Then $\mathcal{L}$ is {\rm (}universally\/{\rm )} orderable.
\end{prop}

\begin{proof}
If $\mathcal{L}$ is not (universally) orderable, then it contains
a (contractible) positive loop by \cite[Proposition 4.7]{ChNe3}. 
Any two elements on (the lift of) such a loop 
cannot have disjoint neighbourhoods in the interval topology.
\end{proof}

The proof shows that if $\mathcal{L}$ is not orderable,
then the interval topology on it is not even T0. 
(There exist distinct points that are not topologically distinguishable, 
i.e.\ every interval neighbourhood of one of them contains the other.) 
This is a general property of the interval topology on Legendrians,
independent of its relation to orderability.

\begin{prop}
\label{T0Haus}
The interval topology on an isotopy class of closed Legendrian submanifolds 
is Hausdorff if and only if it is~\rm{T0}.
\end{prop}

\begin{proof}
The proof is based on the following elementary observation (cf. \cite[Lemma 2.2]{ChNe2}).
Assume that $L, L'\subset Y$ are two closed Legendrian submanifolds
such that $L\cap L'=\varnothing$. Then for any Legendrian $L''$
sufficiently close to $L'$ in the $C^1$-topology, there exists 
a contactomorphism~$\phi\in\mathrm{Cont}_0(Y)$ 
such that $\phi(L'')=L'$ and $\mathop{\mathrm{supp}}(\phi)\cap L =\varnothing$.

Suppose now that $L_1$ and $L_2$ are two Legendrians that do not have disjoint interval neighbourhoods.
Let $L_1^\pm$ be any two Legendrians such that $L_1^-\llcurly L_1\llcurly L_1^+$ and $L_1^\pm\cap L_2=\varnothing$.

For every $\eps>0$, the interval $I:=(L_1^-, L_1^+)$ must intersect the interval 
$(\tau_{-\eps}(L_2),\tau_{\eps}(L_2))$, where $\tau_{t}$, $t\in\R$, is the Reeb flow 
of any contact form. Hence, $L_1^-\llcurly\tau_\eps(L_2)$ and $\tau_{-\eps}(L_2)\llcurly L_1^+$.

If $\eps>0$ is small enough, we can find contactomorphisms $\phi_\pm$ 
such that $\phi_\pm(\tau_{\mp\eps}(L_2))=L_2$ and $\mathop{\mathrm{supp}}(\phi_\pm)\cap L_1^\pm =\varnothing$.
It follows from the invariance of $\llcurly$ that $L_1^-\llcurly L_2\llcurly L_1^+$ and therefore $L_2\in I$.

The Legendrians $L_1^\pm$ disjoint from $L_2$ can be chosen as close 
to $L_1$ in the $C^\infty$-topology as we wish by Example~\ref{PosInJets}
and a general position argument. Thus, $I$ can be made arbitrarily small
in the interval topology and therefore $L_2$ lies in every interval neighbourhood of~$L_1$.
\end{proof}

The non-distinguishable Legendrians obtained from a positive loop
in a non-orderable isotopy class can be chosen disjoint 
by Example~\ref{PosInJets}.
As a partial converse to Proposition~\ref{Haus2order}, we use
an argument very similar to the one in the proof of Proposition~\ref{T0Haus}
to show that disjoint Legendrians in an orderable class are always
separated by intervals. 

\begin{prop}
\label{disjoint}
Let $\mathcal{L}$ be an orderable Legendrian isotopy class. 
If $L_1, L_2\in\mathcal{L}$ are closed Legendrian submanifolds that are disjoint as sets, 
then they have disjoint neighbourhoods for the interval topology on~$\mathcal{L}$.
\end{prop}

Let us emphasise that this falls short of proving
that `orderable' implies `Hausdorff' because of the condition
that $L_1\cap L_2=\varnothing$.

\begin{proof}
Suppose that $L_1$ and $L_2$ do not have disjoint interval neighbourhoods.
In particular, for every $\eps>0$, there exists a Legendrian $L\in\mathcal{L}$ such that
$$
\tau_{-\eps}(L_j)\llcurly L \llcurly \tau_{\eps}(L_j), \quad j=1,2,
$$
where  $\tau_{t}$, $t\in\R$, is the Reeb flow of a contact form.
Hence, there exist positive Legendrian isotopies connecting $\tau_{-\eps}(L_1)$ to $\tau_{\eps}(L_2)$ 
and $\tau_{-\eps}(L_2)$ to $\tau_{\eps}(L_1)$.
Since $L_1\cap L_2=\varnothing$, it follows that for a sufficiently small $\eps>0$, 
we can apply contactomorphisms supported near $L_1$ and $L_2$ to these isotopies 
and obtain positive Legendrian isotopies connecting $L_1$ to $L_2$ and $L_2$ to $L_1$. 
Hence, $L_1\llcurly L_2$ and $L_2\llcurly L_1$,
which is impossible because $\mathcal{L}$ is orderable.
\end{proof}

\begin{rem}
\label{neighbsize}
The same argument shows that if $L_2$ does not intersect a given neighbourhood of $L_1$
in the contact manifold, then it cannot lie in the interval 
$\bigl(\tau_{-\eps}(L_1),\tau_{\eps}(L_1)\bigr)$
for any sufficiently small~$\eps>0$. (Assuming, of course, that the Legendrian 
isotopy class of $L_1$ is orderable.)
\end{rem}

\subsection{Interval topology on contactomorphisms}
\label{ITopCont}
Let $\mathcal{C}$ be a connected component of the contactomorphism group 
of a contact manifold $(Y,\ker\alpha)$.
The {\it interval topology\/} on $\mathcal{C}$ is defined by the family 
of {\it intervals\/} 
$$
I_{a,b} := (a,b):= \{z\in\mathcal{C} \mid a\llcurly z\llcurly b\},\quad a,b\in\mathcal{C}.
$$
The interval topology on the universal cover of $\mathcal{C}$ is defined 
in the same way using the lifted relation.

This topology is invariant with respect to the left and right action
of $\mathrm{Cont}_0(Y,\ker\alpha)$ as well as conjugation by arbitrary 
contactomorphisms (whenever it preserves $\mathcal{C}$).

\begin{prop}
\label{Haus2orderCont}
Let $\mathcal{C}$ be a connected component of the contactomorphism group of a closed
contact manifold. Suppose that the interval topology on the 
{\rm (}universal cover\/{\rm )} of\/ $\mathcal{C}$ is Hausdorff. 
Then $\mathcal{C}$ is {\rm (}universally\/{\rm )} orderable.
\end{prop}

\begin{proof}
If $\mathcal{C}$ is not (universally) orderable, then it contains
a (contractible) positive loop of contactomorphisms by \cite[Criterion 1.2.C]{ElPo}. 
Any two elements on (the lift of) such a loop 
cannot have disjoint neighbourhoods in the interval topology.
\end{proof}

The Hausdorff property is inherited by contactomorphisms from Legendrians 
in the same way as orderability.

\begin{prop}
\label{leg2cont}
Suppose that a contact manifold $(Y,\ker\alpha)$ contains an isotopy class 
of closed Legendrians on {\rm (}the universal cover of\/{\rm )} which the interval topology is Hausdorff.
Then the interval topology is Hausdorff on {\rm (}the universal cover of\/{\rm )} $\mathrm{Cont}_0(Y,\ker\alpha)$.
\end{prop}

\begin{proof}
Let us show that if $\psi\in\mathrm{Cont}_0(Y,\ker\alpha)$ cannot be separated from 
the identity $\id\in\mathrm{Cont}_0(Y,\ker\alpha)$, 
then $\psi(L)$ cannot be separated from~$L$ for any closed Legendrian~$L\subset Y$.
Indeed, for any intervals $I_1, I_2\subset \Leg(L)$ such that $L\in I_1$ and 
$\psi(L)\in I_2$, there exists a small $\eps>0$ such that $\tau_{\pm\eps}(L)\in I_1$
and $\tau_{\pm\eps}(\psi(L))\in I_2$, where $\tau_{t}$, $t\in\R$, is the Reeb flow of the contact form~$\alpha$.
Hence, $I_1\supset (\tau_{-\eps}(L),\tau_{\eps}(L))$ and $I_2\supset (\tau_{-\eps}(\psi(L)),\tau_{\eps}(\psi(L)))$ 
by the transitivity of $\llcurly$. 

Now take $\phi\in \mathrm{Cont}_0(Y,\ker\alpha)$ from the intersection 
of the intervals $(\tau_{-\eps},\tau_{\eps})\ni\id$ and $(\tau_{-\eps}\circ\psi,\tau_{\eps}\circ\psi)\ni\psi$.
Then $\phi(L)\in I_1\cap I_2$. The same argument works {\it mutatis mutandis\/}
for the universal covers.
\end{proof}

\begin{rem}
\label{GraphIndTop}
There is an {\it a priori\/} different way to introduce a topology of
this type on contactomorphisms. Namely, if $\mathcal{C}$ is a connected component 
of the contactomorphism group, the map $\phi\mapsto\Gamma_\phi$ from 
Remark~\ref{LegGraphs} embeds it into a Legendrian isotopy class 
in $Y\times Y\times\R$. The topology induced from the interval topology 
on Legendrians by this embedding is formally rougher than the interval 
topology on~$\mathcal{C}$. 
\end{rem}

\section{$1$-jet bundles and spherical cotangent bundles}
\label{1jetsCoSph}

\subsection{Minimax invariants from generating functions}
\label{minimax}
Let $\Jet^1(L)$ denote the $1$-jet bundle of a closed connected
manifold~$L$ equipped with the standard contact form $du-\lcan$,
where $u$ is the fibre coordinate in $\Jet^0(L)$ and $\lcan$
is the Liouville form on $T^*L$.

Let $\Lambda\subset\Jet^1(L)$ be a Legendrian submanifold.
A function 
$$
S=S(q,\xi):L\times\R^N\to\R
$$ 
is  a {\it generating function\/} for $\Lambda$ if zero is a regular value 
of the partial differential $d_\xi S$ and the map
\begin{equation}
\label{gendef}
\{d_\xi S(q,\xi)=0\}\ni (q,\xi) \longmapsto (q,d_q S(q,\xi),S(q,\xi))\in \Jet^1(L)
\end{equation}
is a diffeomorphism onto~$\Lambda$. A generating function is said to be
{\it quadratic at infinity\/} if $S(q,\xi)=Q(\xi)+\sigma(q,\xi)$, where
$\sigma$ has compact support and $Q(\cdot)$ is a non-degenerate quadratic 
form in the variable~$\xi$.

For a quadratic at infinity function $S:L\times\R^N\to\R$, let
$$
S^c:=\{(q,\xi)\in L\times\R^N\mid S(q,\xi)\le c\}
$$
be its sublevel sets and denote by $S^{-\infty}$ the set $S^c$ for a sufficiently negative~$c\ll 0$.
Following Viterbo~\cite[\S2]{Vi}, one can use homology relative to $S^{-\infty}$ to 
select special critical values~$c_\pm(S)$ of~$S$.

Let $\R^N=V_+\times V_-$ be a decomposition into linear subspaces
such that $Q$ is positive definite on $V_+$ and negative definite on $V_-$.
Consider the relative $\Z/2$-homology classes 
$$
[L\times V_-]\in\HH_{\nu+\dim L}(L\times\R^N,S^{-\infty};\Z/2)
$$
and
$$
[\{q_0\}\times V_-]\in\HH_\nu(L\times\R^N,S^{-\infty};\Z/2),
$$ 
where $\nu=\dim V_-$ and $q_0$ is any point in~$L$. 
Define
$$
c_+(S):=\inf\bigl\{c\in\R\mid [L\times V_-]\in \imath_*\HH_{\nu+\dim L}(S^c,S^{-\infty};\Z/2)\bigr\}
$$
and
$$
c_-(S):=\inf\bigl\{c\in\R\mid [\{q_0\}\times V_-]\in \imath_*\HH_\nu(S^c,S^{-\infty};\Z/2)\bigr\},
$$
where the map $\imath_*:\HH_*(S^c,S^{-\infty};\Z/2)\to \HH_*(L\times\R^N,S^{-\infty};\Z/2)$
of relative homology groups with $\Z/2$ coefficients is induced by the inclusion~$\imath:S^c\to L\times\R^N$.

It follows from the definitions that $c_\pm(S)$ may also be defined by minimax.
Namely,
$$
c_+(S) = \min_{V_+} \max_{L\times\{v_+\}\times V_-} S(q,v_+,v_-)
$$
and 
$$
c_-(S) = \min_{L\times V_+}\; \max_{\{q\}\times\{v_+\}\times V_-} S(q,v_+,v_-).
$$

\begin{exm}
\label{graphs}
Let $\Lambda^f:=\{(q,df(q),f(q))\mid q\in L\}\subset\Jet^1(L)$ be the graph
of the $1$-jet of a smooth function $f:L\to\R$. Then
$$
c_-(S)=\min_L f\quad\text{ and }\quad c_+(S)=\max_L f
$$
for any quadratic at infinity generating function $S:L\times\R^N\to\R$ of the Legendrian
submanifold~$\Lambda^f\subset\Jet^1(L)$. In particular, $c_+(S)=c_-(S)=c$ 
if $S$ generates the graph of the 1-jet of the constant function $f\equiv c$. 
\end{exm}

\begin{lem}[{cf.~\cite[Corollary 2.3]{Vi}}]
\label{cpmconst}
Let $S$ be a quadratic at infinity generating function 
of a closed connected Legendrian submanifold $\Lambda$.
If~$c_+(S)=c_-(S)=c$, then $\Lambda=\Lambda^c=\{(q,0,c)\mid q\in L\}$
is the graph of the $1$-jet of the constant function~$f\equiv c$.
\end{lem}

\begin{proof}
For every $q\in L$, it follows from the assumption and the minimax
characterisation of $c_\pm$ that
$$
c=c_+(S)\ge \min_{V_+}\; \max_{\{q\}\times\{v_+\}\times V_-} S(q,v_+,v_-) \ge c_-(S)=c.
$$
Thus, $c$ is a critical value of the restriction of $S$ to each
fibre $\{q\}\times\R^N$. A point at which it is attained is
a critical point of~$S$. Hence, $\Lambda\supseteq\Lambda^c$ 
by formula~\eqref{gendef}. Since $\Lambda$ and $\Lambda^c$ are 
closed connected submanifolds of the same dimension, this implies 
that~$\Lambda=\Lambda^c$.
\end{proof}

By Chekanov's theorem~\cite{Che},
for any Legendrian isotopy  $\{\Lambda_t\}_{t\in [0,1]}$ of the zero 
section in~$\Jet^1(L)$, there exists a smooth family of quadratic at 
infinity generating functions $S_t:L\times\R^N\to\R$ for $\Lambda_t$.
Furthermore, this family is unique up to stabilisations and fibrewise
diffeomorphisms by the Viterbo--Th\'eret theorem~\cite{Vi, Th1, Th2}.
Therefore, for a Legendrian submanifold $\Lambda$ in the Legendrian
isotopy class of the zero section, one can define
$$
c_\pm(\Lambda):=c_\pm(S)
$$
for any quadratic at infinity generating function of $\Lambda$ 
obtained by Chekanov's theorem.

\subsection{Interval topology on $\Leg(\Lambda^0)$}
\label{ZeroSect}
On the $1$-jet bundle $\Jet^1(L)$ of a closed connected manifold~$L$,
let
$$
\tau_r(q,p,u):=(q,p,u+r)
$$
be the shift by $r\in\R$ in the $u$-direction, 
i.e.\ the time-$r$ map of the Reeb flow for the standard contact form.
Every Legendrian submanifold $\Lambda\subset\Jet^1(L)$ is obviously 
contained in the interval between $\tau_{-r}(\Lambda)$ and 
$\tau_{r}(\Lambda)$ for any $r>0$. Another trivial observation
is that if $S$ is a generating function for $\Lambda$, then
$S+r$ is a generating function for $\tau_r(\Lambda)$. In particular,
we have $c_\pm(\tau_r(\Lambda))=c_\pm(\Lambda)+r$ for any 
$\Lambda\in\Leg(\Lambda^0)$.

\begin{thm}
\label{0sect}
The interval topology is Hausdorff on the Legendrian isotopy class
of the zero section in~$\Jet^1(L)$.
\end{thm}

\begin{proof}
Suppose that $\Lambda_1,\Lambda_2\in \Leg(\Lambda^0)$ do not have disjoint open neighbourhoods
for the interval topology. Applying a global contact isotopy, we may assume that $\Lambda_1$ 
is the zero section itself. For every $\eps>0$, there must then exist a Legendrian 
$\Lambda\in \Leg(\Lambda^0)$ such that
$$
\tau_{-\eps}(\Lambda^0)\llcurly \Lambda\llcurly \tau_{\eps}(\Lambda^0)
$$
and
$$
\tau_{-\eps}(\Lambda_2)\llcurly \Lambda \llcurly \tau_{\eps}(\Lambda_2).
$$
By~\cite[Lemma~5.2]{ChNe1}, the minimax invariants $c_\pm$ are non-decreasing
along a non-negative Legendrian isotopy. Hence,
$$
-\eps \le c_\pm(\Lambda)\le \eps
$$
and
$$
c_\pm(\Lambda_2)-\eps \le c_\pm(\Lambda) \le c_\pm(\Lambda_2)+\eps.
$$
It follows that
$$
-2\eps \le c_\pm(\Lambda_2) \le 2\eps
$$
for all~$\eps>0$. Thus,
$$
c_\pm(\Lambda_2) = 0
$$
and $\Lambda_2$ coincides with the zero section by Lemma~\ref{cpmconst}.
\end{proof}

\begin{cor}
\label{ContHausJet}
The interval topology is Hausdorff on the identity component 
$\mathrm{Cont}_0\bigl(\Jet^1(L),\ker(du-\lcan)\bigr)$ for every closed manifold~$L$.
\end{cor}

\begin{proof}
Follows from Theorem~\ref{0sect} and Proposition~\ref{leg2cont}.
\end{proof}

\begin{exm}
\label{JetQuot}
{\bf (i)} Let $Y\cong T^*L\times S^1$ be the quotient of $\Jet^1(L)$ 
by the $\Z$-action generated by~$\tau_1$ with the 
contact form induced by $du-\lcan$. 
The Reeb flow of this form on $Y$ is periodic.
Hence, the interval topology is not Hausdorff on $\mathrm{Cont}_0(Y)$.
However, it is not hard to deduce from Theorem~\ref{0sect} 
and the covering homotopy theorem that the interval topology 
is Hausdorff on the universal cover of the Legendrian isotopy class 
of the projection of the zero section to $Y$ and therefore on 
$\widetilde{\mathrm{Cont}_0}(Y)$.

\smallskip
\noindent
{\bf (ii)} If one is willing to consider one-dimensional contact manifolds,
it is possible to take $L=\{\mathrm{pt}\}$ in $\mathrm{(i)}$. Then $\Jet^1(L)=\R$
and $Y=S^1$. A~co-oriented contact structure in this dimension is just an orientation.
A~connected Legendrian submanifold is a point and its Legendrian
isotopy class is the connected component of the ambient manifold.
The relation $\llcurly$ defines the usual order on $\R$ and the {\it cyclic\/} order on~$S^1$.  
Hence, the interval topology on $S^1$ is clearly non-Hausdorff but it 
becomes Hausdorff on the universal cover. 
\end{exm}

Generating functions methods developed in~\cite{Vi} and adapted to the
contact case in~\cite{ChNe1} may also be used to show that the interval
topology extends the topology of uniform convergence on `potentials',
i.e.\ on smooth functions corresponding to Legendrian graphs in~$\Jet^1(L)$.

\begin{prop}
\label{IntC0}
The interval topology on $\Leg(\Lambda^0)$ induces the topology
of uniform convergence on the space of smooth functions on~$L$
via the embedding 
$$
C^\infty(L)\ni f \longmapsto \Lambda^f \in \Leg(\Lambda^0).
$$
\end{prop}

\begin{proof}
If $f,g\in C^\infty(L)$, then $f\le g$ pointwise on $L$ if
(and, obviously, only if) $\Lambda^f\lle\Lambda^g$ by~\cite[Corollary~5.4]{ChNe1}.
Hence, $|f-g|<\eps$ on $L$ if and only if $\tau_{-\eps}(\Lambda^g)\llcurly\Lambda^f\llcurly\tau_\eps(\Lambda^g)$,
and the result follows.
\end{proof}

\begin{rem}
Shelukhin~\cite{Sh} used the Hofer distance functional~\cite{Ho} 
to define a norm on $\mathrm{Cont}_0(Y,\ker\alpha)$.
This norm is not conjugation invariant (in that case it would
have to be discrete~\cite{FrPoRo}) but it defines an invariant 
topology on $\mathrm{Cont}_0(Y,\ker\alpha)$ by~\cite[Lemma~10]{Sh}. 
The associated analogue of Chekanov's metric~\cite{Che2}  
considered by Rosen and Zhang~\cite{RoZh}
defines a $\mathrm{Cont}_0(Y,\ker\alpha)$-invariant topology
on Legendrian isotopy classes in~$Y$. 
Proposition~\ref{IntC0} implies that those topologies coincide
with the interval topology `infinitesimally'. 
However, their global behaviour seems to be different. 
For instance, Shelukhin's  norm is non\nobreakdash-degenerate
for every closed contact manifold and therefore the
topology defined by it is always Hausdorff.
\end{rem}

\subsection{Interval topology on $\Leg(ST^*_{\{\mathrm{pt}\}}M)$}
As in~\cite[\S 6]{ChNe1} and~\cite{CFP}, let us combine the results 
of~\S\ref{ZeroSect} for $L=S^{n-1}$ with the `hodograph' contactomorphism 
$$
\Jet^1(S^{n-1})\overset{\cong}{\longrightarrow} ST^*\R^n
$$
taking the zero section to a fibre. 

\begin{cor}
\label{strn}
The interval topology is Hausdorff on the Legendrian isotopy class of the fibre
of\/ $ST^*\R^n$.
\end{cor}

\begin{lem}
\label{cover}
Let $\widetilde M$ be a connected smooth cover of a manifold $M$ with $\dim M\ge 2$.
Assume that the interval topology is Hausdorff on the Legendrian isotopy class
of the fibre of $ST^*\widetilde M$. Then the same holds for~$ST^*M$.
\end{lem}

\begin{rem}
Taking $M=S^1$ and $\widetilde M=\R$ shows that assuming $\dim M\ge 2$ 
is necessary, cf.\ Example~\ref{JetQuot}(ii).
\end{rem}

\begin{proof}[Proof of the lemma]
Let $p:\widetilde M\to M$ be the covering map and $P:ST^*\widetilde M\to ST^*M$ 
the induced projection of the spherical cotangent bundles. Fix a contact form 
$\alpha$ defining the standard contact structure on $ST^*M$. 
Then $\widetilde\alpha=P^*\alpha$ is a contact form defining 
the standard contact structure on $ST^*\widetilde M$. 
The Reeb flows $\tau_s$ and $\widetilde\tau_s$ associated 
to $\alpha$ and $\widetilde\alpha$ satisfy 
$P\circ \widetilde\tau_s=\tau_s$, $s\in\R$.

Pick a point $x\in M$ and let $\widetilde x_j$, $j\ge 1$, be its pre-images in $\widetilde M$.
Set $F=ST_x^*M$ and $\widetilde F_j=ST_{\widetilde x_j}^*\widetilde M$. 
Since $\Leg(\widetilde F_1)$ is orderable by Proposition~\ref{Haus2order},
there does not exist a non-constant non-negative Legendrian 
loop based at~$\widetilde F_1$. Furthermore, there does not exist a non-negative
isotopy from $\widetilde F_j$ to $\widetilde F_k$ in $ST^*\widetilde M$
for $j\ne k$ because a contactomorphism interchanging these
two fibres (induced by a diffeomorphism of $\widetilde M$ interchanging
their base points) would map it to a non-negative isotopy from $\widetilde F_k$ 
to $\widetilde F_j$ and the concatenation of the two isotopies 
would be a non-constant non-negative loop in the fibre class.

Suppose now that the interval topology is not Hausdorff on $\Leg(F)$.
By Proposition~\ref{T0Haus}, there exists a Legendrian $L\ne F\in\Leg(F)$
contained in every interval of the form $\bigl(\tau_{-\eps}(F),\tau_\eps(F)\bigr)$, $\eps>0$.
In other words, for every $\eps>0$, there exists a positive Legendrian
isotopy from $\tau_{-\eps}(F)$ to $\tau_\eps(F)$ passing through~$L$.
By the covering homotopy theorem, this isotopy lifts to a positive Legendrian isotopy
from $\widetilde\tau_{-\eps}(\widetilde F_1)$ to $\widetilde\tau_\eps(\widetilde F_k)$ 
for some~$k\ge 1$ passing through a Legendrian lift $\widetilde L_\eps$ of $L$. 

Note first that if $k\ne 1$ for small enough~$\eps$, 
then since $\widetilde F_k\cap \widetilde F_1=\varnothing$,
it would follow from the argument in the proof of Proposition~\ref{T0Haus} 
that there is a positive isotopy from $\widetilde F_1$ to $\widetilde F_k$, 
which is impossible. 
Hence, $\widetilde L_\eps\in \bigl(\widetilde\tau_{-\eps}(\widetilde F_1),\widetilde\tau_\eps(\widetilde F_1)\bigr)$
for all sufficiently small~$\eps$.
Using Remark~\ref{neighbsize}, we see that $\widetilde L_\eps$ must then intersect a fixed small
neighbourhood of $\widetilde F_1$ in $ST^*\widetilde M$.
However, there are only finitely many lifts of $L$ to $ST^*\widetilde M$
with that property. Thus, at least one such lift is contained 
in arbitrarily small interval neighbourhoods of the fibre~$\widetilde F_1$, 
which contradicts the assumption of the lemma.
\end{proof}

From Lemma~\ref{cover} and Corollary~\ref{strn}, we obtain the following
(potential) improvement of the orderability result in~\cite[Corollary~6.2]{ChNe1}.

\begin{thm}
\label{stcovrn}
The interval topology is Hausdorff on the Legendrian isotopy class
of the fibre of $ST^*M$ for any manifold $M$ smoothly covered
by an open subset of\/~$\R^n$, $n\ge 2$.
\end{thm}

\begin{rem}
\label{dim2and3}
The theorem applies to every surface other than $S^2$ and $\R\mathrm{P}^2$
and to every compact three-manifold other than a quotient of $S^3$ by 
a finite group of isometries of the standard round metric, 
see the discussion starting at the bottom of p.\,1321 in~\cite{ChNe1}.
\end{rem}

Proposition~\ref{leg2cont} shows now that the interval topology 
on contactomorphisms is Hausdorff for another class of contact manifolds.

\begin{cor}
The interval topology is Hausdorff on $\mathrm{Cont}_0(ST^*M)$ for any manifold $M$ 
smoothly covered by an open subset of\/~$\R^n$, $n\ge 2$.
\end{cor}

\section{Relation to Lorentz geometry}

\subsection{Causality and Alexandrov topology}
A {\it spacetime\/} is a connected Lorentz manifold $(\XX,\langle\text{ },\!\text{ }\rangle)$ equipped
with a time-orientation, that is, a continous choice of the future hemicone
$$
C^{\uparrow}_x\subset \{v\in T_x\XX\mid \langle v,v\rangle\ge 0, v\ne 0\}
$$
in the cone of non-spacelike vectors at each point~$x\in\XX$.
(We are assuming that the Lorentz metric has signature $(+,-,\dots,-)$
so that $\langle v,v\rangle>0$ for timelike vectors and $\langle v,v\rangle<0$ for spacelike vectors.)
The vectors in $C^\uparrow_x$ are called {\it future-pointing}.
A piecewise smooth curve in $\XX$ is {\it future-directed\/}
if all its tangent vectors are future-pointing.

The {\it causality relation\/} $\le$  on $\XX$ is defined
by setting $x\le y$ if either $x=y$ or there is a future-directed
curve connecting $x$ to~$y$. 
The {\it chronology relation\/} $\ll$ is defined similarly
by writing $x\ll y$ if there is a future-directed timelike curve 
connecting $x$ to~$y$.

$\XX$ is called {\it causal\/} if it does not contain closed 
future-directed curves. (This is equivalent to requiring
that $\le$ is a partial order.) $\XX$ is {\it strongly 
causal\/} if every point has an arbitrarily small neighbourhood 
such that every future-directed curve intersects it at most once.

The {\it Alexandrov topology\/} on $\XX$ is the interval topology
associated to the chronology relation. (It is named after Alexander D.\ Alexandrov
and must not be confused with the Alexandrov topology on posets
named after Pavel S.\ Alexandrov.) This topology was introduced
by Kronheimer and Penrose who proved the following result
(see \cite[Theorem 4.24]{Pe} or~\cite[Proposition 3.11]{BEE}).

\begin{prop}
\label{AlexHaus}
The Alexandrov topology on a spacetime~$\XX$ is Hausdorff if and only
if~$\XX$ is strongly causal. In that case the Alexandrov
topology coincides with the manifold topology on~$\XX$.
\end{prop}

Let us point out that assuming {\it strong\/} causality is important here. 
For instance, the Alexandrov topology is {\it not\/} Hausdorff 
on the causal spacetime shown in~\cite[Figure 38]{HaEl}.
(The points on the dashed null geodesic do not have disjoint
interval neighbourhoods.) However, in contrast to Proposition~\ref{T0Haus},
the Alexandrov topology in that example is nevertheless~T0 and even~T1.
At the same time, causality may be equivalent to strong causality
under additional compactness assumptions, see e.g.~\cite{BeSa2}.

\begin{rem}
It is unclear whether a useful analogue of strong causality 
can be defined for a Legendrian isotopy class~$\mathcal{L}$. 
Such a definition would require a background topology on~$\mathcal{L}$
playing the role of the manifold topology on~$\XX$. 
That topology should be Hausdorff but not `too fine' 
compared to the interval topology, lest the notion become vacuous. 
For instance, it is easy to see (e.g.\ by considering wavefronts with swallowtails) 
that {\it no\/} Legendrian isotopy class can be `strongly orderable' 
with respect to the $C^k$-topology for any $k\ge 0$.
\end{rem}  

\subsection{Null geodesics, skies, and contact geometry}
\label{Skies}
Suppose now that $\XX$ is a {\it globally hyperbolic\/} spacetime,
i.e.\ it is strongly causal and the causal segments $\{z\in\XX\mid x\le z\le y\}$
are compact for all $x,y\in\XX$. By the Bernal--S\'anchez smooth
splitting theorem~\cite{BeSa1}, a globally hyperbolic spacetime
is foliated by smooth spacelike Cauchy (hyper)surfaces, where a 
{\it Cauchy surface\/} is a subset of a spacetime such that every
inextendible future-directed curve intersects it exactly once.

The {\it space of null geodesics\/} of $\XX$ is the set $\mathfrak N_\XX$ 
of equivalence classes of inextendible future-directed null geodesics 
up to an orientation preserving affine reparametrisation. 
This space carries a canonical structure of a contact 
manifold contactomorphic to the spherical cotangent bundle 
of a Cauchy surface in~$\XX$, see e.g.\ \cite[pp.~252--253]{NT}
or \cite[\S\S 1-2]{ChNe4}.  

The set $\mathfrak S_x\subset\mathfrak N_\XX$ of all null geodesics 
passing through a point $x\in\XX$ is a Legendrian sphere in $\mathfrak N_\XX$
called the {\it sky\/} (or the {\it celestial sphere\/}) of that point, see \cite[\S 4]{ChNe1}.
Since $\XX$ is connected, all skies lie in the same Legendrian isotopy class.
For any Cauchy surface $M\subset\XX$, the associated contactomorphism 
$\rho_M:\mathfrak N_\XX\overset{\cong}{\longrightarrow} ST^*M$ 
takes the sky of a point $x\in M$ to the fibre $ST_x^*M$
and so maps the Legendrian isotopy class of skies to
the Legendrian isotopy class of the fibre.

A conformal (or, equivalently, causal~\cite{Ma}) isomorphism $f:\XX\to\XX'$
maps null pregeodesics to null pregeodesics \cite[Lemma 9.17]{BEE}.
Hence, it induces a contactomorphism $f_*:\mathfrak{N}_\XX\to \mathfrak{N}_{\XX'}$
such that $f_*(\mathfrak{S}_x)=\mathfrak{S}_{f(x)}$.

The map $x\mapsto\mathfrak S_x$ is compatible with the
relations $\ll$ and $\le$ on $\XX$ and $\llcurly$ and $\lle$ 
on the Legendrian isotopy class of skies. For $\le$ and $\lle$,
this was pointed out in~\cite{ChNe1} and elaborated upon in~\cite{BIL}.
 
\begin{prop}
The Legendrian isotopy $\mathfrak S_{\beta(t)}$ corresponding to
a curve $\beta:(a,b)\to\XX$ is non-negative {\rm (}respectively, positive{\rm )} if and
only if that curve is future-directed {\rm (}respectively, future-directed timelike{\rm )}.
\end{prop}

\begin{proof}
For the sake of completeness, we give a short proof 
by computation using the notation and formulas from~\cite{ChNe4}. 
(See~\cite[\S 4]{ChNe1} for a geometric argument 
and \cite[Corollary 3]{BIL} for another computation.
Note that both references used the opposite convention for 
the signature of the Lorentz metric, which `reversed' the
relations.) 

Let $\ell_t:S\to \mathfrak N_\XX$, $t\in (a,b)$, be a parametrisation
of the Legendrian isotopy $\mathfrak S_{\beta(t)}$. (Here $S$
denotes the sphere of dimension $\dim\XX-2$.)
Given $(t_0,\zeta)\in (a,b)\times S$, let
$$
\mathbf{v}=\mathbf{v}(t_0,\zeta):={\left.\frac{d}{dt}\right|}_{t=t_0} \ell_t(\zeta)\in T\mathfrak N_\XX
$$
and choose a family of future-directed null geodesics $\gamma_s:(-1,1)\to\XX$, $s\in(-\eps,\eps)$,
so that $\gamma_s(0)=\beta(t_0+s)$ and the maximal extension of $\gamma_s$
represents the equivalence class $\ell_{t_0+s}(\zeta)$ in $\mathfrak N_\XX$.
The Jacobi vector field of this family at the point $x=\beta(t_0)=\gamma_0(0)$
is then 
\begin{equation}
\label{Jacobi}
J(x)={\left.\frac{d}{ds}\right|}_{s=0} \gamma_s(0)=\dot\beta(t_0).
\end{equation}
Take a spacelike Cauchy surface $M$ passing through the point $x=\beta(t_0)$
and let $\alpha_M$ be the associated contact form on $\mathfrak N_\XX$.
Plugging~\eqref{Jacobi} into the formula for $\alpha_M(\mathbf{v})$ on p.~381 in \cite{ChNe4}, 
we obtain
\begin{equation}
\label{evalalpha}
\alpha_M(\mathbf{v}) = \frac{\bigl\langle \dot\gamma_0(0),\dot\beta(t_0)\bigr\rangle}{\bigl\langle \dot\gamma_0(0), n_M(x)\bigr\rangle},
\end{equation}
where $n_M$ is the future-pointing unit normal vector to~$M$.

A vector in a time-oriented Lorentz vector space is future-pointing (respectively, future-pointing timelike)
if and only if its scalar product with every future-pointing null
vector is non-negative (respectively, positive). Thus, the denominator in~\eqref{evalalpha} 
is positive. Furthermore, since $\ell_{t_0}:S\to\mathfrak N_\XX$ parametrises 
the sky of $x=\beta(t_0)$, the tangent vector $\dot\gamma_0(0)$ 
runs through all null directions in $C^\uparrow_x$ as $\zeta\in S$ varies.
Hence, $\alpha_M(\mathbf{v})\ge 0\; (\text{respectively, }>0)$ for all $\mathbf{v}=\mathbf{v}(t_0,\zeta)$ 
if and only if $\dot\beta(t_0)$ is future-pointing (respectively, future-pointing timelike).
\end{proof}

\begin{cor}
\label{caus2order}
$x\le y\Longrightarrow \mathfrak S_x\lle\mathfrak S_y$ and $x\ll y\Longrightarrow \mathfrak S_x\llcurly\mathfrak S_y$
\end{cor}

The converse implications do not hold for certain (somewhat special)
spacetimes, see e.g.~\cite[Example 10.5]{ChNe2}.
This problem does not occur if the Legendrian isotopy 
class of skies is orderable.

\begin{prop}
\label{order2caus}
Suppose that the Legendrian isotopy class of skies in $\mathfrak N_\XX$
is orderable.
Then $ \mathfrak S_x\lle\mathfrak S_y\Longrightarrow x\le y$ and $\mathfrak S_x\llcurly\mathfrak S_y\Longrightarrow x\ll y$.
\end{prop}

\begin{proof}
The claim about $\le$ and $\lle$ was proved in~\cite[Proposition 1.3]{ChNe3}.
If~$\mathfrak S_x\llcurly\mathfrak S_y$, then $\mathfrak S_{x'}\llcurly\mathfrak S_y$ for
all points $x'$ in a small neighbourhood of $x$ in~$\XX$ because
the sky depends smoothly on the point and $\llcurly$ is open in
the smooth topology. So $x'\le y$ for all such points by
the previous case. Hence, we can pick $x'$ so that $x\ll x'$
and $x'\le y$. This implies that $x\ll y$ by \cite[Proposition 2.18]{Pe}.
\end{proof}

\begin{cor}
\label{Int2Alex}
Suppose that the Legendrian isotopy class of skies in $\mathfrak N_\XX$
is orderable. The interval topology on this Legendrian isotopy class induces the
usual manifold topology on $\XX$ via the embedding $x\mapsto\mathfrak S_x$.
\end{cor}

\begin{proof}
It follows from Corollary~\ref{caus2order} and Proposition~\ref{order2caus}
that the interval topology induces the Alexandrov topology on $\XX$.
A globally hyperbolic spacetime is strongly causal, so the Alexandrov
topology coincides with the manifold topology by Proposition~\ref{AlexHaus}.
\end{proof}

\begin{rem}
The Legendrian isotopy class of skies or, equivalently, the Legendrian
isotopy class of the fibre of $ST^*M$ for a Cauchy surface $M\subset\XX$
is orderable if the universal cover $\widetilde M$ is non-compact by \cite[Remark 8.2]{ChNe2} 
or the integral cohomology ring of $\widetilde M$ is not isomorphic to
that of a compact rank one symmetric space by \cite[Theorem 1.2]{FrLaSch}
combined with~\cite[Proposition 4.7]{ChNe3}.
In the remaining cases, one can use the fact that this Legendrian
isotopy class is always {\it universally\/} orderable by~\cite[Theorem~1.1]{ChNe3}
and obtain a substitute for Corollary~\ref{Int2Alex} 
by considering the map $\widetilde x\mapsto\widetilde{\mathfrak S}_{\widetilde x}$ 
from the (finite) universal cover $\widetilde\XX$ of the spacetime $\XX$ 
to the universal cover of the Legendrian isotopy class of skies 
in $\mathfrak N_{\widetilde\XX}$, cf.~\cite[\S 1.2]{ChNe3}. 
\end{rem}

\subsection{Interval completion of a globally hyperbolic spacetime}
Let now $\XX$ be a globally hyperbolic spacetime such that

\medskip
\noindent
$(*)$\hspace{0.06\textwidth}
\parbox{0.8\textwidth}{the interval topology is Hausdorff on the Legendrian\break 
isotopy class in $\mathfrak{N}_\XX$ containing the skies of points in~$\XX$.}

\medskip
\noindent
In particular, this Legendrian isotopy class is orderable by Proposition~\ref{Haus2order}.
Theorem~\ref{stcovrn} guarantees that condition $(*)$ is satisfied if a
smooth spacelike Cauchy surface $M\subset\XX$ is smoothly covered by 
an open subset of $\R^n$, and Remark~\ref{dim2and3} shows that this 
assumption is not too restrictive for $(2+1)$- and $(3+1)$-dimensional spacetimes.

Let us define the {\it interval completion\/} of $\XX$ by setting
$$
\widehat\XX := \overline{\{\mathfrak{S}_x \mid x\in\XX\}},
$$
where the closure is taken with respect to the interval topology
on the Legendrian isotopy class of skies in~$\mathfrak N_\XX$.
In other words, a point in $\widehat\XX$ is a Legendrian sphere 
in the space of null geodesics such that there is a sequence 
of skies of points in~$\XX$ converging to it in the interval
topology. The {\it interval boundary\/} of $\XX$ is the
difference 
$$
\p\XX := \widehat\XX - \{\mathfrak{S}_x \mid x\in\XX\}.
$$

The definition of $\widehat\XX$ and the results collected 
in the previous subsection have the following immediate
consequences:

\begin{enumerate}
\item $\widehat\XX$ is a Hausdorff topological space.
\item The map $x\mapsto\mathfrak{S}_x$ is an open embedding of $\XX$ into $\widehat\XX$.
\item The relations $\llcurly$ and $\lle$ on $\widehat\XX$ restrict to $\ll$ and $\le$ on $\XX$.
\item Every causal isomorphism $f:\XX\to\XX'$ extends to a causal isomorphism $\widehat{f}:\widehat{\XX}\to\widehat{\XX'}$.
\end{enumerate}

Another basic corollary is that a point of $\widehat\XX$ lies in $\p\XX$
if (and, obviously, only if) it doesn't have either a past or a future in~$\XX$.

\begin{prop}
If $L\in\p\XX$, then at most one of the sets
$$
I_\XX^+(L):=\{x\in\XX\mid L\llcurly\mathfrak{S}_x\}
\quad\text{and}\quad
I_\XX^-(L):=\{x\in\XX\mid \mathfrak{S}_x\llcurly L\}
$$
is nonempty.
\end{prop}

\begin{proof}
Suppose that there exist $x_\pm\in I_\XX^\pm(L)$.
Let $x_n$, $n\to\infty,$ be a sequence of points in $\XX$
such that their skies converge to $L$ in the interval
topology. Then $\mathfrak{S}_{x_-}\llcurly \mathfrak{S}_{x_n}\llcurly \mathfrak{S}_{x_+}$
for all large $n$. Hence, $x_-\ll x_n\ll x_+$ by Proposition~\ref{order2caus}
and therefore $x_n$ is contained in the compact causal segment
$\{z\in\XX\mid x_-\le z\le x_+\}$. Thus, there is a subsequence
converging to a point $x\in\XX$ . The corresponding sequence
of skies will converge to $\mathfrak{S}_x\ne L$, which
contradicts the uniqueness of limit in a Hausdorff space.
\end{proof}

Finally, let us show that the interval boundary differs from the classical causal boundary~\cite[\S 6.8]{HaEl}
and from the boundary defined by Low~\cite[\S 6]{Lo3} already in the simplest example.

\begin{prop}
Let $\XX=\R^{1,n}$ be the flat Minkowski spacetime.  
Then $\widehat\XX=\XX$ and $\p\XX=\varnothing$.
\end{prop}

\begin{proof}
Consider the Cauchy surface $M=\{0\}\times\R^n\subset\R^{1,n}$ and the associated
contactomorphism
$$
\rho_M:\mathfrak{N}_{\R^{1,n}}\overset{\cong}{\longrightarrow} ST^*\R^n.
$$
By definition, $\rho_M$ maps the sky of a point $(t,y)\in \R^{1,n}$ 
to the fibre $ST_y^*\R^n$ for $t=0$ and to the Legendrian lift of the 
$(n-1)$-sphere $S(y,|t|)=\{y'\in\R^n \mid \|y'-y\|=|t|\}$ co-oriented inwards for $t<0$ and outwards for~$t>0$. 
If $(t,y)\to\infty$ in $\R^{1,n}$, then a straightforward 
computation shows that the minimax invariants $c_\pm$ 
of the image of $\rho_M(\mathfrak{S}_{(t,y)})$ under 
the hodograph contactomorphism
$$
ST^*\R^n \overset{\cong}{\longrightarrow} \Jet^1(S^{n-1})
$$
satisfy $|c_+|+|c_-|\to\infty$, see~\cite[\S 6]{ChNe1}. 
Hence, the skies of such points cannot be contained 
in any fixed interval by the monotonicity of $c_\pm$. 
\end{proof}

\end{document}